\documentclass[14pt]{article}
\usepackage[english]{babel}
\usepackage{amsfonts,amssymb,amsmath,amstext,amsbsy,amsopn,amscd,amsthm,graphicx,euscript,graphicx}
\usepackage{graphics}
\newtheorem{thm}{Theorem}
\newtheorem{lm}{Lemma}
\newtheorem{prop}{Proposition}
\newtheorem{crl}{Corollary}
\newtheorem{st}{Statement}

\newcounter{tdfn}
\newenvironment{dfn}
{\vspace{0.15cm}{\bf Definition \arabic{tdfn}.}} {\par
\addtocounter{tdfn}{1}}
\newcounter{trk}
\newenvironment{rk}
{\vspace{0.15cm}{\bf Remark \arabic{trk}.}} {\par
\addtocounter{trk}{1}}
{\endtrivlist}

\def\:{\colon}

\def\R{{\mathbb R}}
\def\Z{{\mathbb Z}}
\def\0{{\mathbf 0}}
\def\1{{\mathbf 1}}

\def\C{{\mathbb C}}
\def\R{{\mathbb R}}

\title{Invariants of classical braids valued in $G_{n}^{2}$}

\author{V.O.Manturov}

\date{}

\begin{document}

\maketitle


\begin{abstract}
The aim of the present note is to enhance groups $G_{n}^{3}$ and to construct new invariants of classical braids. In particular, we construct invariants valued in $G_{N}^{2}$ groups. In groups $G_{n}^{2}$, the identity problem is solved, besides, their structure is much simpler than that of $G_{n}^{3}$.
\end{abstract}



\section{Introduction}
In the paper \cite{Great}, the author defined a family of groups $G_{n}^{k}$ depending on two natural numbers $n>k$ and formulated the following principle:

{\em If a dynamical system describing a motion of  $n$ particles, possesses a good generic property of codimension $1$ governed by exactly  $k$ particles, then this dynamical system has invariants valued in $G_{n}^{k}$.}

We shall not formulate the definition of good property in its full generality.

In  \cite{MN}, a partial case of this general principle was calculated explicitly: when considering a motion of $n$ pairwise distinct points on the plane and choosing the generic codimension $1$ property to be ``three points are collinear'' (for $k=3$), we get a homomorphism from the $n$-strand pure braid group $PB_{n}$ to the group $G_{n}^{3}$.  This allows one to get powerful invariants of classical braids.
In \cite{HigherGnk}, applications of groups  $G_{n}^{k}$ for $k>3$ to the study of fundamental groups of other configuration spaces are given.

 Let us define the group $`G_{n}^{3},n\ge 3$, which generalises $G_{n}^{3}$. The group $`G_{n}^{3}$ is given by a presentation having $3{n \choose 3}$ generators $a'_{ijk}$ which are indexed by triples of pairwise distinct numbers $i,j,k\in \{1,\cdots, n\}$ defined up to the order reversal. Thus, $a'_{123}=a'_{321}\neq a'_{132}$. The group $`G_{n}^{3}$ is given as:

$$`G_{n}^{3}=\langle a'_{ijk}|(1),(2),(3)\rangle,$$

where

$$(a'_{ijk})^{2}=1, i,j,k\in \{1,\cdots,n\} \mbox{ pairwise distinct };\;\;\;\eqno(1) $$
$$a'_{ijk}a'_{pqr}=a'_{pqr}a'_{ijk}, i,j,k,p,q,r\in \{1,\cdots, n\},Card\{i,j,k\})=Card(\{p,q,r\})=3;\eqno(2)$$

$$ Card(\{i,j,k\}\cap \{p,q,r\})\le 1 \;\;\;$$
 $$a'_{ijk}a'_{ijl}a'_{ikl}a'_{jkl}=a'_{jkl}a'_{ikl}a'_{ijl}a'_{ijk}, i,j,k,l\in\{1,\cdots, n\} \mbox{ pairwise distinct}.\eqno(3)$$.

The group $G_{n}^{3}$ can be obtained from $`G_{n}^{3}$ by identifying those generators, whose triples of indices coincide as sets: the generator $a_{ijk}$ of the group $G_{n}^{3}$ is equal to the image of any of $a'_{ijk},a'_{jik},a'_{ikj}$.

Note that $`G_{3}^{3}$ is the free product of three groups  $\Z_{2}$ and the group $`G_{4}^{3}$ has no relations of type (2), since any two subsets of cardinality 3 of the set $\{1,2,3,4\}$ have intersection of cardinality at least $2$. When considering the groups $G_{n}^{3}$ one usually requires $n>3$; however, the groups $`G_{3}^{3}$ are already interesting.

Later, we shall also need the group $G_{n(n-1)}^{2}$, however, in its definition (unlike that of $G_{N}^{2}$, see, e.g., \cite{Bardakov,Great}) the indices $p,q$  of generators $a_{p,q}$  will be not elements of the set $\{1,\cdots, n(n-1)\}$ but rather ordered pairs of distinct elements from  $1$ to $n$. Hence, for example, we have generators $a_{12,34},a_{12,31}$. The relations are standard: we have $a_{p,q}^{2}=1$ for each generator $a_{p,q}$ as well as $a_{p,q}a_{r,s}=a_{r,s}a_{p,q}$ for pairwise distinct ordered $p,q,r,s$ and $(a_{p,q}a_{p,r}a_{q,r})^{2}=1$ for pairwise distinct $p,q,r$.

For example, $a_{12,23}$ commutes with $a_{21,24}$, since $\{1,2\}$ are $\{2,1\}$
distinct ordered sets.

For our paper, the following lemma is crucial.
\begin{lm}
The map $\phi:`G_{n}^{3}\to G_{n(n-1)}^{2}$, which takes $a'_{ijk} \mapsto a_{ij,ik} a_{kj,ki},$ is well defined.\label{l1}
\end{lm}

First note that the two terms $a_{ij,ik}$ and $a_{kj,ki}$ commute, which yields $\phi(a'_{ijk})=\phi(a'_{kji})$ that we need, since $a'_{kji}=a'_{ijk}$.

This lemma follows from a direct check. We shall check the most interesting case:
$$(a'_{ijk}a'_{ijl}a'_{ikl}a'_{jkl})^{2}\mapsto (a_{ij,ik}a_{kj,ki}a_{ij,il}a_{lj,li}a_{ik,il}a_{lk,li}a_{jk,jl}a_{lk,lj})^{2}=$$

$$=(a_{ij,ik}a_{ij,il}a_{ik,il})^{2}(a_{lj,li}a_{lk,li}a_{lk,lj})^{2}a_{jk,jl}^{2}a_{kj,ki}^{2}=1.$$

Note that this lemma has its own importance. For the group $G_{N}^{2}$ we have a
simple minimality criterion (see, e.g., \cite{Coxeter}): a word $g$ in standard presentation of $G_{N}^{2}$ has minimal length if and only if no word ${\tilde g}$ equivalent to $g$ by means of exchange relations $a_{p,q}a_{p,r}a_{q,r}\mapsto a_{q,r}a_{p,r}a_{p,q}$ and commutativity relation $a_{pq}a_{rs}\mapsto a_{rs}a_{pq}$ contains two equal letters $a_{p,q}a_{p,q}$ in order (in our case $N=n(n-1)$ each of the letters $p,q,r,s$ itself has two indices). Hence, we get a {\em sufficient minimality condition} for words from $`G_{n}^{3}$: if the image $\phi(\alpha)$ is minimal, then $\alpha$ itself is minimal.

The author does not know whether the map $\phi$ is injective; this is an important open problem. Another problem is whether one can construct some map analogous to the map $\phi$ from $G_{n}^{3}$ to some $G_{M(n)}^{2}$ (for sufficiently large $M(n)$ depending on $n$). Their positive solution might shed light on the word problem for groups $G_{n}^{k}$ for $k>2$.

\section{Construction of the main invariant}

Let us now construct the map $f$, which maps a pure braid $\beta\in PB_{n}$ ($n\ge 3$) to an element from $`G_{n}^{3}$. We shall deal with pure braids, with points in the initial and final moments uniformly distributed over the unit circle: $z_{j}(0)=exp(\frac{2\pi j}{n})$. By a {\em braid} we mean a set of smooth functions  $\beta(t)=\{z_{1}(t),\cdots,z_{n}(t)\}, t\in [0,1]$ valued in $\C^{1}=\R^{2}$ such that $\beta(0)=\beta(1)$ coincides with the set of values mentioned above, and all $z_{i}(t)$ are pairwise distinct for any $t$. By a {\em critical moment} we mean such a value of $t$ for which there are some three indices  $i,j,k$ such that $z_{i}(t),z_{j}(t),z_{k}(t)$ are collinear.
We say that the braid is {\em good and stable} if:
\begin{enumerate}
\item the number of critical values of this braid is finite;
\item for each critical moment  $t$ there exist exactly one triple of indices $(i,j,k)$ for which $z_{i}(t),z_{j}(t),z_{k}(t)$ are collinear;
\item (stability) the number of critical moments does not change after any small perturbation of the braid.
\end{enumerate}

Every braid can be made good and stable by an arbitrarily small perturbation.

With a good and stable pure braid $\beta$ we naturally associate a word in generators $a'_{ijk}$: with each critical moment $t_{l}$ with three corresponding collinear points with indices $(i_{l},j_{l},k_{l})$ with  $j_{l}$ in the middle, we associate the generator $a'_{i_{l},j_{l},k_{l}}$. The word $f(\beta)$ is the product of all generators corresponding to critical moments as  $t$ grows.

\begin{thm}
The map $\beta\mapsto f(\beta)$ constructed above is a homomorphism from $PB_{n}$ to $`G_{n}^{3}$\label{beta}.
\end{thm}

\begin{rk}
For braids which are not pure, the method mentioned above also defines a map from braids to elements of $PB_{n}$, if, for example, we fix positions in the initial and terminal moment, but do not fix their order. However, this map is not a homomorphism.
\end{rk}

The proof of this theorem is essentially the same as that of the theorem from \cite{MN} about the map from $PB_{n}$ to $G_{n}^{3}$. We use the standard principle \cite{Great}, saying that in order to study braid isotopy, it suffices to consider singularities of codimension two. As codimension one singularities (triples of collinear point) give rise to generators, codimension two singularities give rise to relations, namely, we get the following list of codimension two singularities (see \cite{MN}):

\begin{enumerate}
\item unstable triple point which disappears after a small perturbation; this corresponds to the relation ${a'_{ijk}}^{2}=1$;

\item coincidence of two moments when two independent triple points appear. This corresponds to the relations
$a'_{p}a'_{q}=a'_{q}a'_{p}$, where triples $p,q$ have no more than one index in common;

\item four collinear points; this gives rise to the relation where on the left hand side we have a product of four generators, and on the right hand side we have a product of the same generators in the reverse order. Otherwise we can represent this relation as a ``cyclic'' deformation of the dynamics consisting of eight elementary deformations.

    The above four generators are all possible triples of indices among the four given indices, for example,  $a'_{ijk},a'_{ijl},a'_{ikl},a'_{jkl}$. The only novelty in comparison with  \cite{MN} is that the relation in  $`G_{n}^{3}$ is more exact than that in $G_{n}^{3}$ considered in \cite{MN}. Considering the line passing through a quadruple of points and assigning numbers $i,j,k,l$ to these points as we encounter them along the line, one can directly check that in the cyclic deformation the generator $a'_{ijk}$ (corresponding to three collinear points with point number $j$ in the middle) can not be next to the generator $a'_{ikl}$. Thus, we get some cycle, which can be obtained from the LHS of (3) by a cyclic permutation and order reversal. There are eight such words.

    Since the square of each of the generators is equal to the unit, the order reversal gives rise to the inverse word.
\end{enumerate}

With the product of pure braid we associate the product of words by definition.

From the above, we get the Proof of Theorem \ref{beta}.

The group $PB_n$ of pure   $n$-strand braid group can be given by the following
generators and relations:

$$b_{ij}b_{kl}=b_{kl}b_{ij},\quad  i<j<k<l \mbox{ or } i<k<l<j,$$
$$b_{ij}b_{ik}b_{jk} = b_{ik}b_{jk}b_{ij} = b_{jk}b_{ij}b_{ik}, \quad i<j<k,$$
$$b_{jl}b_{kl}b_{ik}b_{jk}=b_{jl}b_{kl}b_{ik}b_{jk},\quad i<j<k<l.
$$

For any distinct indices $i,j$, $1\le i,j\le n$ we define the element ${c'}_{i,j}$ of $`G_{n}^3$ as$${c'}_{i,j}=\prod_{k=j+1}^n a'_{j,i,k}\cdot \prod_{k=1}^{j-1} a'_{j,i,k}$$
(the product is taken over all $k\neq i,k\neq j$).

\begin{prop}
$$f(b_{ij})\mapsto {c'}_{i,i+1}^{-1}\dots {c'}_{i,j-1}^{-1}{c'}_{j,i}^2 {c'}_{i,j-1}\dots {c'}_{i,i+1}, i<j.$$
\end{prop}

\begin{proof}
Let us consider the $n$-point configuration $z_k(0)=e^{2\pi ik/n}, k=1,\dots,n$ on the plane $\mathbb R^2=\mathbb C$ (all points lie on the same circle $C=\{z\in\mathbb C\,|\,|z|=1\}$).

For each $i<j$ the braid $b_{ij}$
can be represented by the following dynamical system:
\begin{enumerate}
\item the point $i$ moves along the interior side of the circle $C$, passes by the points $i+1, i+2,\dots, j-1$ and land on the circle before the point $j$ (Fig.~\ref{fig:bij_moves} upper left);
\item the point $j$ moves over the point $i$ (Fig.~\ref{fig:bij_moves} upper right);
\item the point $i$ returns to its initial position over the points $j,j-1,\dots, i+1$ (Fig.~\ref{fig:bij_moves} lower left);
\item the point $j$ returns to its position (Fig.~\ref{fig:bij_moves} lower right).
\end{enumerate}

 \begin{figure}
  \centering
  \begin{tabular}{cc}
    \includegraphics[width=0.3\textwidth]{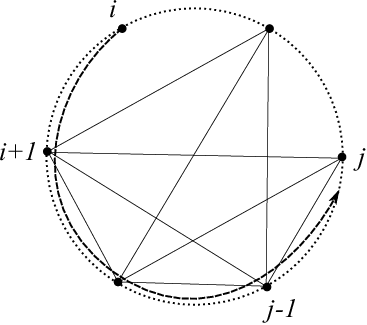} &
    \includegraphics[width=0.3\textwidth]{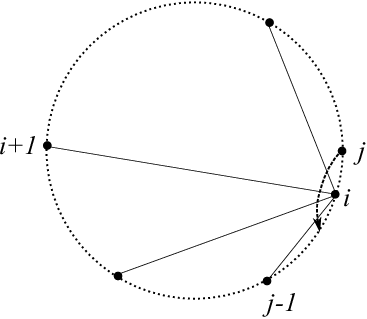} \\
    \includegraphics[width=0.3\textwidth]{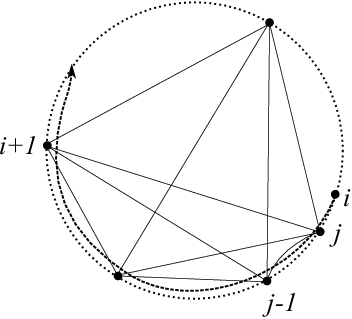} &
    \includegraphics[width=0.3\textwidth]{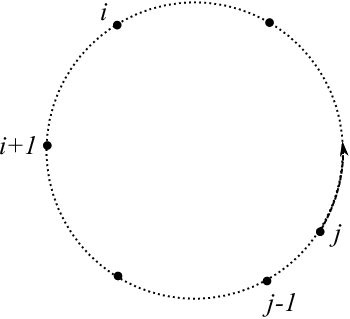}
  \end{tabular}
  \caption{Dynamical system corresponding to $b_{ij}$}\label{fig:bij_moves}
 \end{figure}

As we check all the situations in the dynamical systems where three points lie on the same line, and write down these situations as letters in a word of the group $G_{n}^{3}$, we exactly get the element $${c'}_{i,i+1}^{-1}\dots {c'}_{i,j-1}^{-1}{c'}_{i,j}^2 {c'}_{i,j-1}\dots {c'}_{i,i+1}.$$
\end{proof}

\begin{crl}
The map $\Phi=\phi\circ f$ is a homomorphism $PB_{n}\to G_{n(n-1)}^{2}$.
\end{crl}
As we see from the above construction, the map  $\Phi$ can be constructed without use of the auxiliary group $`G_{n}^{3}$: for a general position dynamics we can directly write down the product of two elements from $G_{n(n-1)}^{2}$.

The constructed map $f$ is a generalisation of quite a powerful invariant of classical braids constructed in $\cite{MN}$. The composition  $\Phi=\phi\circ f$ can be easily calculated for braids. For instance, for the generator of the pure three-strand braid group where the first and the second points are fixed  and the third point goes around the second one, we get the product of the four generators of the group  $G_{n(n-1)}^{2}$: two moments
corresponding to collinear triples of points in different orders, give rise to a word of length $4$ in $`G_{3}^{3}=\Z_{2}*\Z_{2}*\Z_{2}$.

Let us now construct a homomorphism from $`G_{n}^{3}$ to an automorphism of a free product of some copies of  $\Z_{2}$, which is ``spiritually'' similar to the Hurwitz action of the braid group on the free group.

Namely, we set the image of $g\sim g(a'_{ijk})$ to be  $g: a_{ij}\mapsto a_{ik}a_{ij}a_{ik},a_{kj}\mapsto a_{ki}a_{kj}a_{ki},
a_{m}\mapsto a_{m}, \mbox{where } m\neq \{ij\},\{kj\}$

\begin{thm}
The map $g$ is a well defined homomorphism.
\end{thm}

\begin{proof}
The statement follows from a direct check of relations of $`G_{n}^{3}$. One can mention some similarity of the main relation (3) and the third Reidemeister move. The proof is similar to the proof of lemma \ref{l1}.

The action ${a'}_{ijk}^{2}$ conjugates the element $a_{ij}$
twice by $a_{ik}$, hence, the element $a_{ij}$ remains unchanged; the same is true for $a_{kj}$ and $a_{ki}$.

The
``far commutativity'' $a'_{p}a'_{q}=a'_{q}a'_{p}$ for $|p\cap q|\le 1$
follows from the fact that pairs of indices taking part in
conjugations coming from  $p$, are distinct from pairs of indices taking
part in conjugations coming from $q$.

For the relation $a'_{ijk}a'_{ijl}a'_{ikl}a'_{jkl}=a'_{jkl}a'_{ikl}a'_{ijl}a'_{ijk}$, the LHS and the RHS have
the same action:
$a_{ij}\mapsto a_{il}a_{ik}a_{ij}a_{ik}a_{il}, a_{ik}\mapsto a_{il}a_{ik}a_{il}, a_{jk}\mapsto a_{jl}a_{jk}a_{jl}$,
analogously one calculates the action of $a_{kj},a_{lj},a_{lk}$.

\end{proof}
\section{One more generalisation of groups $G_{n}^{k}$}

Now, let us give one more geometric construction, which gives rise to invariants of classical braids. As the constructions, presented above, it relies upon {\em general position codimension 1 property}.

Namely, we consider a motion of  $n$ points $z_{j}(t),j=1,\cdots, n$ inside the unit circle $D=\{|z|< 1\}=\{x,y,x^{2}+y^{2}<1\}$. One can readily check that through any two points inside the circle one can draw exactly two circles tangent to the absolute $|z|=1$. We are interested in those moments, where

{\em some three points $z_{i}(t),z_{j}(t),z_{k}(t)$ lie on a circle tangent to the absolute.}

Later on, when considering the circle tangent to the absolute, we shall enumerate points on this circle starting from the tangency point in the counterclockwise direction. We shall say that the point $a$ {\em precedes} the point $b$ if when passing the circle starting from the point $X$ counterclockwise, we first encounter $a$ and then $b$.

\begin{figure}
\centering\includegraphics[width=200pt]{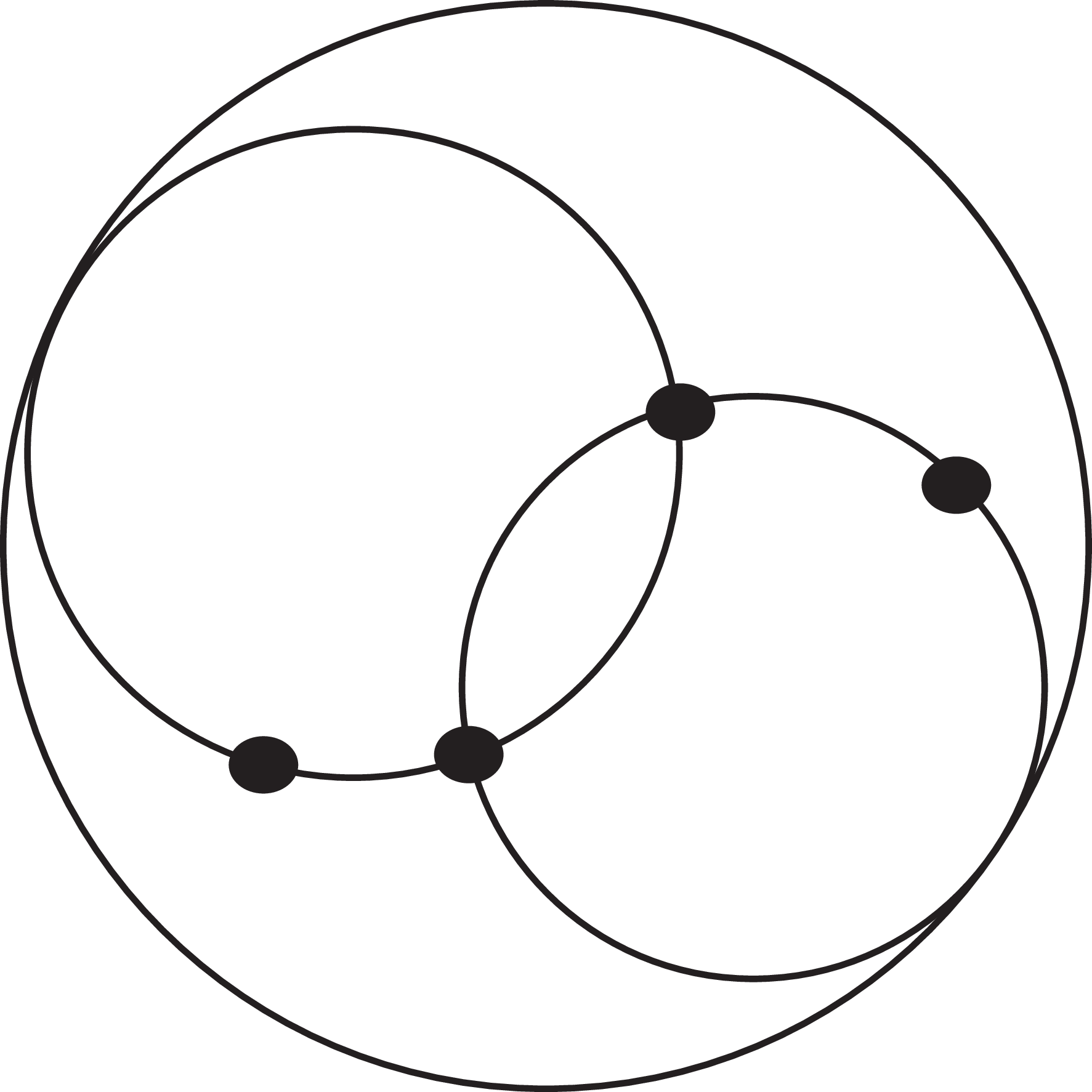}
\caption{Two circles tangent to the absolute}
\label{}
\end{figure}

Unlike the $k$-properties considered in \cite{Great,MN}  (``three points are collinear'' for $k=3$, ``four points are on the same circle'' for $k=4$), this property has the following subtlety. There are two circles passing through the two points and tangent to the absolute.
More precisely, this disadvantage can be formulated as follows:

{\em from the fact that three points $a,b,c$
belong to a circle tangent to the absolute
and three points $a,b,d$ belong to a circle tangent to the absolute, it does not follow that all four points $a,b,c,d$ belong
to a circle tangent to the absolute.
}

One can easily see that the following statement holds.
\begin{st} If points $a,b,c$ belong to the same circle tangent to the absolute and points $a,b,d$ belong to the same circle tangent to the absolute and in both circles $a$ precedes $b$, then  $a,b,c,d$ belong to the same circle tangent to the absolute.\label{stm}
\end{st}

The ability to work with such situations is quite important, since when studying dynamical systems of motions of several points generic codimension 1 properties can be related to more complicated curves than just lines or circles, so that there are a fixed number of curves of the given type passing through fixed $k-1$ points.
This leads to the generalisation of the $G_{n}^{k}$-approach and to the following group (in our case $k=3$).
Net let us fix the number of points $n$.

\begin{dfn}
The group $``G_{n}^{3}$ is

$$``G_{n}^{3}=\langle a''_{ijk}|(4),(5),(6)\rangle,$$
having $n(n-1)(n-2)$ generators $a''_{ijk}$, where $i,j,k$ range all possible {\em ordered} triples of points from $1$ to $n$, and three types of relations:

$$ {a''_{ijk}}^{2}=1 \mbox{ for all pairwise distinct } i,j,k \;\;\;\;\;\eqno{(4)}$$

$$a''_{ijk}a''_{pqr}=a''_{pqr}a''_{ijk},\;\;\;\;\;\; \eqno{(5)}$$
if none of the ordered pairs $\{i,j\},\{i,k\},\{i,l\}$ coincides with any of the ordered pairs $\{p,q\},\{p,r\},\{q,r\}$.

$$(a''_{ijk}a''_{ijl}a''_{ikl}a''_{jkl})^{2}=1, p,q,r,s\mbox{ are pairwise distinct. }\;\;\;\;\ \eqno{(6)}$$
\end{dfn}

Note that the condition (5) significantly differs from (2). For example, $a''_{123}$  commutes with $a''_{421}$,
but does not commute with
 $a''_{134}$.

There is a natural homomorphism from $`G_{n}^{3}$ to $G_{n}^{3}$ which forgets ``primes'' and forgets the order of indices of generators. In the case of $``G_{n}^{3}$, there is no such evident homomorphism. Indeed, in  the group $``G_{n}^{3}$ for $n\ge 4$ one has, for example, the relation $a''_{123}a''_{421}=a''_{421}a_{123}$; the relation $a_{123}a_{124}=a_{124}a_{123}$ in $G_{n}^{3}$ does not hold. Moreover, if we take the quotient of the group $G_{n}^{3}$ by all commutativity relations of such sort, the resulting quotient group will be isomorphic to a direct product of groups $\Z_{2}$.

Now let dynamics of motion of  $n$ points inside the unit circle be given. Imposing natural ``general position'' conditions with respect to the circles tangent to the absolute, we can associate a word in letters $a''_{ijk}$ to this dynamics, as follows. At each critical moment $t_{l}$ we have three points on the circle $(i_{l},j_{l},k_{l})$, which are enumerated in the counterclockwise direction starting from the tangency point. With such a moment, we associate the generator $a''_{i_{l},j_{l},k_{l}}$. The word $g(\beta)$ is the
product of all generators corresponding to all critical moments, as  $t$ increases.

Analogously to Theorem \ref{beta} and the main theorem of  \cite{Great}, one can prove
\begin{thm}
The map $\beta\mapsto g(\beta)$ constructed above, is a homomorphism from the pure braid group $PB_{n}$ to $``G_{n}^{3}$\label{gamma}.
\end{thm}

\begin{proof}
Analogously to Theorem \ref{beta}, let us enumerate all events of codimension 2 which will lead us to relations. The case of ``unstable triple point'' is completely analogous to the ``unstable triple point'' case from \ref{beta}: at some moment, we have three points on a circle tangent to the absolute, and this disappears after a small perturbation of dynamics. This corresponds to the relation $``a_{ijk}^{2}=1$.

In the same manner, one can deal with ``quadruple points''. Let our dynamics be such that at some moments some four points are on the same circle tangent to the absolute. Let us enumerate these points in the counterclockwise direction starting from the tangency point: $i,j,k,l$. After a small perturbation, the
quadruple point splits into four instances with triple points and corresponding generators $a''_{ijk},a''_{ijl},a''_{ikl},a''_{jkl}$. The ``opposite'' small perturbation gives rise to the product of the same four generators in the inverse order in such a way that $a''_{ijk}$ can not be next to $a''_{ikl}$.

The most important difference between the case of  $``G_{n}^{3}$ and the case of $`G_{n}^{3}$ is the situation with two ``independent simultaneous'' triples of points.

Assume that for some moment  $t$ for the  $n$ points, there are two triples of indices  $m=\{a,b,c\},m'=\{d,e,f\}$,
such that $(z_{a}(t),z_{b}(t),z_{c}(t))$ belong to a circle tangent to the absolute and
 $(z_{d}(t),z_{e}(f),z_{g}(t))$ belong to another circle tangent to the absolute and in each triple the indices are pairwise distinct.If
 $Card((a,b,c)\cap (d,e,f))\le 1$, then we get a commutativity relation (5) analogous to the relation (2). If the intersection consists of two indices, we get (unordered) triples $(a,b,c)$ and $(a,b,f)$, then, according to Statement \ref{stm}, for the same triple, the point $a$ precedes  $b$ when counting counterclockwise, and in the second triple, $b$ precedes $a$, which is also described by relation (5).

\end{proof}

The group $``G_{n}^{3}$ admits a natural homomorphism to the group $G_{N}^{2}$, where  $N=n(n-1)$,
and each of the indices $p,q$ of generators $a_{p,q}$ is an ordered pair of distinct elements from  $1$ to $n$.

Let us define the map $h:``G_{n}^{3}\to G_{N}^{2}$ by the formula
$$
h:a''_{ijk}\mapsto a_{ij,ik}.\;\;\;\;\eqno{(7)}
$$

\begin{thm}
The map constructed above is a homomorphism $h:``G_{n}^{3}\to G_{N}^{2}$.
\end{thm}
The check is left to the reader.

As in the case of $`G_{n}^{3}$,
we get a sufficient minimality
condition for elements from $``G_{n}^{3}$.

\begin{rk}
The composite map $PB_{n}\mapsto G_{N}^{2}$
can be constructed without mentioning the group
 $``G_{n}^{3}$.
Namely, when considering the braid
$\alpha\in PB_{n}$ as a dynamics of motion of  $n$ points inside
the unit circle, to each moment when some three points  $i,j,k$ belong to the same circle tangent to the absolute,
we associate the product $a_{ij,ik}a_{ij,ik}a_{ik,jk},$
if when walking along the circle starting from the tangent point, we encounter
these points in the order $i,j,k$.
\end{rk}

The author is grateful to I.M.Nikonov and D.A.Fedoseev  for various fruitful discussions.


\begin{thebibliography}{100}

\bibitem{Bardakov} V.G.Bardakov, The Virtual and Universal Braids, {\em Fundamenta Mathematicae}, 184 (2004), P. 1-18.


\bibitem{Coxeter} V.O.Manturov, {On groups $G_{n}^{2}$ and Coxeter groups}, Russ.Math. Surv., vol. 72 (2017),
{\bf 2}, pp. 234-235,
{\em
http://arxiv.org/abs/1512.09273}


\bibitem{Great} V.O.Manturov, Non-Reidemeister Knot Theory and Its Applications in Dynamical Systems, Geometry, and Topology, {\em
http://arxiv.org/abs/1501.05208}


\bibitem{HigherGnk} V.O.Manturov, The Groups $G_{n}^{k}$ and fundamental groups of configuration spaces, {\em J. Knot Theory \& Ramifications, 2017, Vol. 26.}

\bibitem{MN} V.O.Manturov, I.M.Nikonov,  {\em On braids and groups $G_{n}^{k}$},
{J. Knot Theory \& Ramifications, 2015, Vol. 24, No.13.}


\end{thebibliography}
 \end{document}